\documentclass[12pt]{amsart}

\usepackage{amsmath}
\usepackage{amsfonts}
\usepackage{amssymb}
\usepackage{amsthm}
\usepackage{graphics}
\usepackage{amscd}
\usepackage{graphicx,epsfig}
\usepackage{color}
\usepackage[all]{xy}
\usepackage{url}
\usepackage[margin = 1in] {geometry}

\DeclareMathOperator{\tr}{tr}
\DeclareMathOperator{\Div}{div}

\renewcommand{\epsilon}{\varepsilon}

\newcommand{\boM}{\mathcal{M}}

\newcommand{\boA}{\mathcal{A}}

\newcommand{\boB}{\mathcal{B}}

\newcommand{\boN}{\mathcal{N}}

\newcommand{\R}{\mathbb{R}}

\newcommand{\B}{\mathbb{B}}

\renewcommand{\S}{\mathbb{S}}

\newcommand{\eps}{\varepsilon}

\newcommand{\Ome}{\Omega}

\newcommand{\barre}[1]{\overline{#1}}
\renewcommand{\phi}{\varphi}

\newtheorem{Thm}{Theorem}[section]
\newtheorem{Prop}{Proposition}[section]

\newtheorem{Cor}{Corollary}[section]

\newtheorem*{thm*}{Theorem}

\theoremstyle{remark}
\newtheorem{rmk}{Remark}

\newcounter{remark}

\newcounter{case}

\newcounter{construction}

\newcounter{fact}

\DeclareMathOperator{\bric}{BRic}
\DeclareMathOperator{\ric}{Ric}
\DeclareMathOperator{\vol}{Vol}

\title{Stable minimal hypersurfaces in $\R^6$}
\author{Laurent Mazet}
\address{Institut Denis Poisson, CNRS UMR 7013, Universit\'e de Tours, 
Universit\'e d'Orl\'eans, Parc de Grandmont, 37200 Tours, France}
\email{laurent.mazet@univ-tours.fr}

\begin{document}

\begin{abstract}
Following the strategy developed by Chodosh, Li, Minter and Stryker, and 
using the volume estimate of Antonelli and Xu, we prove that, in $\R^6$, a 
complete, two-sided, stable minimal hypersurfaces is flat.
\end{abstract}

\maketitle

\section{introduction}

A minimal hypersurface $M^n$ of $\R^{n+1}$ is a critical point of the 
$n$-volume 
functional. It is characterized by its vanishing mean curvature. If a unit 
normal vectorfield $\nu$ is defined along $M$ and $\phi$ is a function with 
compact support on $M$, one can consider a deformation of $M$ with initial 
speed $\phi\nu$. The computation of the second derivative of the $n$-volume 
along this deformation at initial time gives
\[
\int_M|\nabla\phi|^2-|A_M|^2\phi^2
\] 
where $A_M$ is the second fundamental form of $M$. So asking that this quantity 
is non negative for any $\phi$ means that $M$ is a minimum at order $2$ of the 
$n$-volume. Such a minimal hypersurface is called stable.

The stable Bernstein problem asks wether a complete stable minimal hypersurface 
is a flat affine hyperplane. We give a positive answer in the case $n=5$.

\begin{Thm}\label{th:main}
Let $M^5\looparrowright \R^6$ be an immersed, complete, connected, two-sided, 
stable 
minimal hypersurface. Then $M$ is a Euclidean hyperplane.
\end{Thm}

A particular class of stable minimal hypersurface is given by minimal graphs 
over $\R^n$. In \cite{Ber}, Bernstein proved that a minimal graph over $\R^2$ 
has 
to be a plane. In the sixties, the same question for higher dimensions was 
studied 
in a series 
of paper by Fleming 
\cite{Fle}, De Giorgi \cite{DeG}, Almgren \cite{Alm2} and Simons \cite{Sims}. 
They proved that minimal graphs over $\R^n$ are planes if $n\le 7$. For $n\ge 
8$, Bombieri, De Giorgi and Giusti \cite{BoGiGi} were able to construct 
counter-examples and gave also in $\R^8$ an example of a stable minimal 
hypersurfaces that is not a hyperplane.

Concerning the stable Bernstein problem, the question was solved positively in 
$\R^3$ by Do Carmo and Peng \cite{DoPe}, Fischer-Colbrie and Schoen \cite{FCSc} 
and Porogelov \cite{Pog} in the early eighties. In higher dimension, Schoen, 
Simon and Yau \cite{ScSiYa,ScSi} were able to settle the stable Bernstein in 
$\R^{n+1}$, $n\le 6$, under a Euclidean volume growth assumption (see also the 
recent work of Bellettini \cite{Bel}).

Recently Chodosh and Li \cite{ChoLi3} were able to answer positively the 
stable Bernstein problem in $\R^4$. Later two alternative proofs came out: one 
by Catino, Mastrolia and Roncoroni \cite{CaMaRo} and one by Chodosh and Li 
\cite{ChoLi}. Actually in \cite{ChoLi}, Chodosh and Li develop a second 
strategy to prove the result. Then, in a joint work Minter and Stryker 
\cite{ChLiMiSt}, they were able to apply this strategy in the case of $\R^5$ to 
solved  
the stable Bernstein problem in this dimension as well.

As in \cite{ChLiSt,ChLiMiSt}, it is well known that Theorem~\ref{th:main} comes 
with 
corollaries like curvature estimates for stable minimal immersions in 
$6$-dimensional manifolds and characterization of finite Morse index minimal 
hypersurfaces in $\R^6$. For example, we have
\begin{Cor}
Let $(X^6,g)$ be a complete Riemannian manifold whose sectional curvature 
satisfies $|sec_g|\le K$. Then any compact, two-sided, stable minimal immersion 
$M^5\looparrowright X$ satisfies
\[
|A_M|(q)\min(1,d_M(q,\partial M))\le C(K)
\]
for $q\in M$.
\end{Cor}

The basic idea to prove Theorem~\ref{th:main} is to obtain a Euclidean growth 
estimate for  the volume of $M$ and then apply the work of Schoen, Simon and 
Yau. The strategy of Chodosh and Li is 
a way towards this estimate. We refer to \cite{ChoLi,ChLiMiSt} for a good 
presentation of their ideas. Let us give some elements. Let $M$ be a stable 
minimal hypersurface in $\R^{n+1}$ with induced metric $g$. Inspired by the 
work of Gulliver and Lawson \cite{GuLa}, they consider the conformal metric 
$\tilde g=r^{-2}g$ where $r$ is the Euclidean distance to $0$ in $\R^{n+1}$. If 
$M$ was a hyperplane passing through the origin $(M\setminus \{0\},\tilde g)$ 
would be isometric to the Euclidean product $\S^{n-1}\times \R$. In the general 
case, the idea of Chodosh and Li is that the stability assumption implies that 
the geometry of $(M\setminus \{0\},\tilde g)$ should look like $\S^{n-1}\times 
\R$. 
In \cite{ChLiMiSt}, the authors consider the bi-Ricci curvature which is a 
certain combination of sectional curvatures. It was introduced by Shen and Ye in
\cite{ShYe}, already to study minimal surfaces (see precise definition in 
Section~\ref{sec:prelim}). Notice that on $\S^{n-1}\times 
\R$, the bi-Ricci curvature is lower bounded by $n-2$. In \cite{ChLiMiSt}, the 
authors prove that the stability of $M$ implies a positive spectral lower bound 
for 
the 
bi-Ricci curvature of $(M\setminus \{0\},\tilde 
g)$. More precisely they prove that, on $(M\setminus \{0\},\tilde g)$, the 
operator $-\widetilde\Delta+(\widetilde\bric_--1)$ is non-negative where 
$\widetilde\bric_-$ is the punctual minimum of the bi-Ricci curvature of 
$\tilde g$. This should be understood as a weak form of $\widetilde\bric \ge 1$.

The second step of the strategy consists in the construction of a $\mu$-bubble 
in $(M\setminus \{0\},\tilde g)$ with a spectral lower bound for its Ricci 
curvature. In some sense, they identify in any sufficiently large part of 
$(M\setminus 
\{0\},\tilde g)$ a hypersurface that play the role of $\S^{n-1}\times\{t\}$ in 
$\S^{n-1}\times 
\R$.

The last step in \cite{ChLiMiSt} is to obtain a Bishop-Gromov volume estimate 
for the $\mu$-bubble under the spectral lower bound on the Ricci curvature. In 
their paper, 
the proof of this volume estimate was specific to dimension $3$. Recently, 
Antonelli and Xu \cite{AnXu} have proved such a Bishop-Gromov estimate in any 
dimension.

Once the $\tilde g$-volume of the $\mu$-bubble is controlled, this gives an 
estimate of its volume in the original metric $g$ and then control the growth 
of the volume of $M$ tanks to an isoperimetric inequality due to Michael and 
Simon \cite{MiSi} and Brendle \cite{Bre2}.

In the present paper, we also follow the above strategy of \cite{ChLiMiSt}. 
Here we 
consider a weighted bi-Ricci curvature $\bric_\alpha$ where the parameter 
$\alpha$ does not give the same weight to all sectional curvature in the 
combination (a similar idea appear in the recent article by Hong and Yan 
\cite{HoYa}).  We prove a spectral lower bound for the weighted bi-Ricci 
curvature: the operator $-a\widetilde 
\Delta+(\widetilde\bric_{\alpha-}-\delta)$ is 
non-negative where $a,\delta\in \R$. At that step, $a,\alpha$ are two 
parameters that should be chosen such that $\delta>0$.

By imposing some new constraints on $a$ and $\alpha$, we are then able to 
construct the $\mu$-bubble with a spectral lower bound on the Ricci curvature. 
At the last step, we apply the Bishop-Gromov estimate of Antonelli and Xu 
\cite{AnXu}. In order to do so, this imposes some new constraints on the 
parameters $a$ and $\alpha$. Nevertheless, the choice $a=\frac{11}{10}$ and 
$\alpha=\frac{40}{43}$ fits all the constraints. The end of the proof then 
follows the line of \cite{ChLiMiSt}.

\subsection*{Organization}
In Section~\ref{sec:prelim}, we fix some notations that we use all along the 
paper. Section~\ref{sec:specbiric} is devoted to the proof the spectral lower 
bound for $\widetilde\bric_\alpha$ for the Gulliver-Lawson metric on a stable 
minimal hypersurface. In Section~\ref{sec:mububble}, we construct the 
$\mu$-bubble with a spectral lower Ricci bound. We end the proof of 
Theorem~\ref{th:main} in Section~\ref{sec:proof}. Along the paper, we specify 
the value of $n$, $a$ and $\alpha$ only when it is necessary, we hope this 
allows 
to understand where the constraints come from.

\subsection*{Acknowledgments} The author was partially supported by the 
ANR-19-CE40-0014 grant. Part of this work was carried out during a stay at 
the Instituto 	de Matem\'aticas de la Unversidad de Granada (IMAG), the 
author would to thank its members for their hospitality.

\section{Preliminaries}\label{sec:prelim}

Let $(M^n,g)$ be a Riemannian manifold and $(e_i)_{1\le i\le n}$ be an 
orthonormal basis of $T_pM$. For $\alpha\in \R$, we recall or define
\begin{itemize}
\item the Ricci curvature $\ric(e_1,e_1)=\sum_{i=2}^nR(e_1,e_i,e_i,e_1)$,
\item the punctual minimum of the Ricci curvature $\lambda(p) =\min_{v\in 
T_pM,|e|=1}\ric(e,e)$,
\item the weighted bi-Ricci or $\alpha$-bi-Ricci curvature 
\[
\bric_\alpha(e_1,e_2)=\sum_{i=2}^nR(e_1,e_i,e_i,e_1)+\alpha\sum_{j=3}^n 
R(e_2,e_j,e_j,e_2)
\]
\item the minimum of the $\alpha$-bi-Ricci curvature 
$\Lambda_\alpha(p)=\min_{(e,f)\text{ orthonormal in }T_pM}\bric_\alpha(e,f)$
\end{itemize}
Notice that for $\alpha=1$, $\bric_1$ is the classical bi-Ricci curvature 
as defined in \cite{ShYe}.

If $\Sigma\looparrowright M$ is a hypersurface with unit normal $\nu$. We use 
the following conventions:
\begin{itemize}
\item the second fundamental form of $\Sigma$ is 
$A_\Sigma(X,Y)=(\nabla_X\nu,Y)=-(\nabla_XY,\nu)$ and
\item the mean curvature of $\Sigma$ is $H=\tr A_\Sigma$.
\end{itemize}

If $\Ome$ is a subset of $M$, we denote by $\boN_{\rho}(\Ome)$ the 
$\rho$-tubular neighborhood of $\Ome$: the set of 
points at distance less than $\rho$ from $\Ome$.

We finish by a simple remark that we use in Section~\ref{sec:specbiric}.
\begin{rmk}\label{rmk:quadratic}
Let $A\in \boM_n(\R)$ be a positive definite symmetric matrix and $B\in 
\R^n$. Then the function $f: X\in \R^n\mapsto X^\top A X+B^\top X\in\R$ 
is 
lower bounded and its minimum is given by $-\frac14 B^\top A^{-1}B$.
\end{rmk}

\section{Spectral lower bound for the weighted bi-Ricci curvature}
\label{sec:specbiric}

Let $F:M^n\looparrowright \R^{n+1}$ be a complete two-sided minimal 
hypersurface and $g$ 
its induced metric. We consider the Gulliver-Lawson conformal metric $\tilde 
g=r^{-2}g$ 
where $r$ is the Euclidean distance function to $0$. Notice that if $F(p)=0$, 
$\tilde g$ is not defined. So we consider $N=M\setminus F^{-1}(0)$. As it was 
observed by Gulliver and Lawson \cite{GuLa}, the metric $(N,\tilde g)$ is 
complete.

The first step of the proof of Theorem~\ref{th:main} consists in proving that 
the stability assumption 
can be translated in a spectral lower bound for the $\alpha$-bi-Ricci curvature 
of the metric $\tilde g$. Actually we have the following result.

\begin{Thm}\label{th:stable_estim}
Let $M^n\looparrowright \R^{n+1}$ be a two-sided stable minimal hypersurface.
Suppose $n=5$, then, for $a=\frac{11}{10}$, $\alpha=\frac{40}{43}$ and 
$\delta=\frac3{10}$, 
there is a smooth function $V$ such that 
\[
V\ge \delta-\widetilde\Lambda_\alpha
\]
and 
\begin{equation}\label{eq:specest}
\int_N |\nabla\phi|_{\tilde g}^2 dv_{\tilde g}\ge \int_N \frac1a V\phi^2 
dv_{\tilde 
g}
\end{equation}
for any $\phi\in C_c^1(N)$.
\end{Thm}

\subsection{Recalling some computations}

We first recall some computations and results of \cite{ChLiMiSt}. 

We denote by $\nu$ the unit normal to $M$ and by $|dr|$ the norm of the 
differential of $r$ along $M$ with respect to the $g$-metric.

Let 
$(e_i)_{1\le i\le n}$ be an orthonormal basis for the metric $g$ then, for the 
conformal metric $\tilde g$, an orthonormal basis is given by $\tilde 
e_i=re_i$. The sectional curvatures of $g$ and $\tilde g$ are related by
\begin{equation}\label{eq:sectional}
\tilde{R}_{ijji}=r^2R_{ijji}+2-|dr|^2-dr(e_i)^2-dr(e_j)^2-(p,\nu)(A_{ii}+A_{jj})
\end{equation}
(see \cite[Proposition~3.5]{ChLiMiSt}).

The second result that we want to recall is the writing of the stability 
inequality in the conformal metric $\tilde g$. We have
\begin{equation}\label{eq:stab}
\int_N |\nabla \phi|^2 dv_{\tilde g}\ge 
\int_N\left(r^2|A|^2-\frac{n(n-2)}2+\frac{n^2-4}4|dr|^2\right)\psi^2dv_{\tilde 
g}
\end{equation}
for any $\phi\in C_c^1(N)$ (see \cite[Proposition~3.10]{ChLiMiSt}).

\subsection{Estimating the curvature terms}
In this subsection, we want to relate the curvature term in the stability 
inequality \eqref{eq:stab} to the $\alpha$-bi-Ricci curvature. We use the 
notations of the preceding subsection.
\begin{Prop}\label{prop:confbiric}
\begin{align*}
\widetilde{\bric}_\alpha(\tilde e_1,\tilde 
e_2)=&r^2\bric_\alpha(e_1,e_2)+2(n-1+\alpha(n-2))-(n+\alpha(n-1))|dr|^2\\
&-((n-2-\alpha)dr(e_1)^2+\alpha(n-3)dr(e_2)^2)\\
&-(p,\nu)((n-2-\alpha)A_{11}+\alpha(n-3)A_{22})
\end{align*}
\end{Prop}
\begin{proof}
Summing \eqref{eq:sectional} and using $\tr A=0$, we have
\begin{align*}
\widetilde{\bric}_\alpha(\tilde e_1,\tilde 
e_2)=&\sum_{i=2}^n\tilde{R}_{1ii1}+\alpha\sum_{j=3}^n\tilde{R}_{2jj2}\\
=&r^2\bric_\alpha(e_1,e_2)+ 2(n-1)-(n-1)|dr|^2-(n-1)dr(e_1)^2\\
&-(|dr|^2-dr(e_1)^2)-(p,\nu)((n-1)A_{11}-A_{11})+2\alpha(n-2)\\
&-\alpha(n-2)|dr|^2-\alpha(n-2)dr(e_2)^2-\alpha(|dr|^2-dr(e_1)^2-dr(e_2)^2)\\
&-\alpha(p,\nu)((n-2)A_{22}-A_{11}-A_{22})\\
=&r^2\bric_\alpha(e_1,e_2)+2(n-1+\alpha(n-2))-(n+\alpha(n-1))|dr|^2\\
&-((n-2-\alpha)dr(e_1)^2+\alpha(n-3)dr(e_2)^2)\\
&-(p,\nu)((n-2-\alpha)A_{11}+\alpha(n-3)A_{22})
\end{align*}
\end{proof}

\begin{Prop}\label{prop:biric}
\[
\bric_\alpha(e_1,e_2)=-\sum_{i=1}^nA_{1i}^2-\alpha\sum_{j=2}^nA_{2j}^2-\alpha 
A_{11}A_{22}
\]
\end{Prop}

\begin{proof}
Applying Gauss formula we have
\begin{align*}
\bric_\alpha(e_1,e_2)=&\sum_{i=2}^n(A_{11}A_{ii}-A_{1i}^2)+\alpha 
\sum_{j=3}^n(A_{22}A_{jj}-A_{2j}^2)\\
=&-\sum_{i=1}^nA_{1i}^2+\alpha(-A_{22}(A_{11}+A_{22})-\sum_{j=3}^nA_{2j}^2)\\
=&-\sum_{i=1}^nA_{1i}^2-\alpha\sum_{j=2}^nA_{2j}^2-\alpha 
A_{11}A_{22}
\end{align*}
\end{proof}

Using the above computation, we obtain the following estimate of the 
curvature term. This estimate introduces some constraints on $\alpha$ and a 
second parameter $a$.

\begin{Prop}\label{prop:secondform}
Let $a,\alpha>0$ such that $a>\frac12$, $2a\ge \alpha$ and
\[
W=(a-\frac12)\big(a-\frac{n-2}{2n}(1+2\alpha)\big)-\frac{n-2}{4n}(1-\alpha)^2>0
\]
Let us define
\[
f=\frac{(n-2)^2}{8W}\Big((a-\frac12)\frac{n-2}n(1+\alpha\frac{n-4}{n-2})^2
 +(1-\alpha)^2(a+\frac{n-2}{2n}-\frac2n\alpha)\Big)
\] 
Then
\[
ar^2|A|^2+f(1-|dr|^2)\ge 
-r^2\bric_\alpha(e_1,e_2)+(p,\nu)((n-2-\alpha)A_{11}+\alpha(n-3)A_{22})
\]
\end{Prop}

\begin{proof}
By Proposition~\ref{prop:biric}, the right-hand side of the expected inequality 
satisfies to
\begin{equation}\label{eq:curvterm2}
\begin{split}
-r^2\bric_\alpha(e_1,e_2)+&(p,\nu)((n-2-\alpha)A_{11}+\alpha(n-3)A_{22})\\
=&r^2\Big(\sum_{i=1}^nA_{1i}^2+\alpha\sum_{j=2}^nA_{2j}^2+\alpha 
A_{11}A_{22}\\
&\qquad+(\frac p{r^2},\nu)((n-2-\alpha)A_{11}+\alpha(n-3)A_{22})\Big)\\
=&r^2\Big(A_{11}^2+\alpha A_{22}^2+\alpha A_{11}A_{22} 
+\sum_{i=2}^nA_{1i}^2+\alpha\sum_{j=3}^nA_{2j}^2\\
&\qquad+(\frac 
p{r^2},\nu)((n-2-\alpha)A_{11}+\alpha(n-3)A_{22})\Big)
\end{split}
\end{equation}
The vector $A_\Delta=(A_{11},\cdots, A_{nn})$ belongs to the sub-space 
$F_n=\{X\in \R^n\mid 
x_1+\cdots+x_n=0\}$. We write a decomposition in an orthonormal basis of 
$F_n$ as
\[
\begin{pmatrix}
A_{11}\\\vdots\\A_{nn}
\end{pmatrix}=\sum_{i=1}^{n-3}\begin{pmatrix}
0\\0\\E_i
\end{pmatrix}x_i + \frac1{\sqrt{2n(n-2)}}\begin{pmatrix}
n-2\\n-2\\-2\\\vdots\\-2
\end{pmatrix}z_1+\frac1{\sqrt 2}\begin{pmatrix}
1\\-1\\0\\\vdots\\0
\end{pmatrix}z_2
\]
where $(E_i)_{1\le i\le n-3}$ is an orthonormal basis of $F_{n-2}$. So we have
\begin{equation}\label{eq:curvterm}
\begin{split}
A_{11}^2+&\alpha A_{22}^2+\alpha A_{11}A_{22}+(\frac 
p{r^2},\nu)((n-2-\alpha)A_{11}+\alpha(n-3)A_{22})\\
=&(\frac{\sqrt{n-2}}{\sqrt{2n}}z_1+\frac{z_2}{\sqrt 
2})^2+\alpha(\frac{\sqrt{n-2}}{\sqrt{2n}}z_1-\frac{z_2}{\sqrt 2})^2+\alpha 
(\frac{n-2}{2n}z_1^2-\frac12 z_2^2)\\
&+(\frac 
p{r^2},\nu)\big(\frac{\sqrt{n-2}}{\sqrt{2n}}(n-2+\alpha(n-4))z_1+\frac{n-2}{\sqrt
2}(1-\alpha)z_2\big)\\
=&\frac{n-2}{2n}(1+2\alpha)z_1^2+\sqrt{\frac{n-2}n}(1-\alpha)z_1z_2+\frac12 
z_2^2\\
&+(\frac 
p{r^2},\nu)\frac{n-2}{\sqrt2}\big(\sqrt{\frac{n-2}{n}}(1+\alpha\frac{n-4}{n-2})z_1+
 (1-\alpha)z_2\big)
\end{split}
\end{equation}
For $a>0$, we are interested in the minimum (if it exists) of
\begin{equation}\label{eq:mino}
\begin{split}
a(z_1^2+z_2^2)-&\frac{n-2}{2n}(1+2\alpha)z_1^2-\sqrt{\frac{n-2}n}(1-\alpha)z_1z_2-\frac12
 z_2^2\\
&-(\frac 
p{r^2},\nu)\frac{n-2}{\sqrt2}(\sqrt{\frac{n-2}{n}}(1+\alpha\frac{n-4}{n-2})z_1+(1-\alpha)z_2)
\end{split}
\end{equation}
The matrix of the quadratic part of the above expression is
\[
\begin{pmatrix}
a-\frac{n-2}{2n}(1+2\alpha)&-\sqrt{\frac{n-2}{4n}}(1-\alpha)\\
-\sqrt{\frac{n-2}{4n}}(1-\alpha)&a-\frac12
\end{pmatrix}
\]
This matrix is positive definite if $a>\frac12$ and its determinant is positive:
\[
W=(a-\frac12)\big(a-\frac{n-2}{2n}(1+2\alpha)\big)-\frac{n-2}{4n}(1-\alpha)^2>0
\]
If it's the case, by Remark~\ref{rmk:quadratic} with vector $B=-(\frac 
p{r^2},\nu)\frac{n-2}{\sqrt2}\big(\sqrt{\frac{n-2}{n}}(1+\alpha\frac{n-4}{n-2}),(1-\alpha)\big)$,
 the quantity in \eqref{eq:mino} is lower bounded by
\begin{align*}
&-(\frac 
p{r^2},\nu)^2\frac{(n-2)^2}{8W}\Big((a-\frac12)\frac{n-2}n(1+\alpha\frac{n-4}{n-2})^2
 +\frac{n-2}{n}(1-\alpha)^2(1+\alpha\frac{n-4}{n-2})\\
&\qquad\qquad\qquad\qquad+(a-\frac{n-2}{2n}(1+2\alpha))(1-\alpha)^2\Big)\\
=&-(\frac 
p{r^2},\nu)^2\frac{(n-2)^2}{8W}\Big((a-\frac12)\frac{n-2}n(1+\alpha\frac{n-4}{n-2})^2
 +(1-\alpha)^2(a+\frac{n-2}{2n}-\frac2n\alpha)\Big)\\
=&-(\frac p{r^2},\nu)^2f
\end{align*}
Since $(\frac pr,\nu)^2=(1-|dr|^2)$, we have then proved that
\begin{align*}
a(z_1^2+z_2^2)+\frac f{r^2}(1-|dr|^2)\ge &
\frac{n-2}{2n}(1+2\alpha)z_1^2+\sqrt{\frac{n-2}n}(1-\alpha)z_1z_2+\frac12 
z_2^2\\
&+(\frac 
p{r^2},\nu)\frac{n-2}{\sqrt2}(\sqrt{\frac{n-2}{n}}(1+\alpha\frac{n-4}{n-2})z_1+(1-\alpha)z_2)
\end{align*}

Combining this with \eqref{eq:curvterm2} and \eqref{eq:curvterm}, if $2a\ge 
\alpha$, we have
\begin{align*}
a|A|^2+\frac f{r^2}(1-|dr|^2)\ge& a (|A_\Delta|^2+\sum_{i\neq 
j}A_{ij}^2)+\frac f{r^2}(1-|dr|^2)\\
\ge & A_{11}^2+\alpha A_{22}^2+\alpha A_{11}A_{22}+(\frac 
p{r^2},\nu)((n-2-\alpha)A_{11}+\alpha(n-3)A_{22})\\
& +a\sum_{i\neq j}A_{ij}^2\\
\ge & A_{11}^2+\alpha A_{22}^2+\alpha A_{11}A_{22}+\sum_{i=2}^nA_{1i}^2+\alpha 
\sum_{j=3}^nA_{2j}^2\\
&+(\frac 
p{r^2},\nu)((n-2-\alpha)A_{11}+\alpha(n-3)A_{22})\\
\ge& 
-\bric_\alpha(e_1,e_2)+\frac{(p,\nu)}{r^2}((n-2-\alpha)A_{11}+\alpha(n-3)A_{22})
\end{align*}
This is the expected estimate
\end{proof}

\subsection{Proof of Theorem~\ref{th:stable_estim}}

Let us assume that the basis is chosen such that $\widetilde 
\Lambda_\alpha=\widetilde 
\bric_\alpha(\tilde e_1,\tilde e_2)$. From 
\eqref{eq:stab}, we are looking for a lower bound for 
$r^2|A|^2-\frac{n(n-2)}2+\frac{n^2-4}4|dr|^2$. Under the assumptions of 
Proposition~\ref{prop:secondform}, $\alpha\le 1$ (such that $n-2-\alpha\ge 
\alpha(n-3)$) and using 
Proposition~\ref{prop:confbiric}, we have
\begin{align*}
a\Big(r^2|A|^2-\frac{n(n-2)}2+&\frac{n^2-4}4|dr|^2\Big)\\
\ge &
-r^2\bric_\alpha(e_1,e_2)+(p,\nu)((n-2-\alpha)A_{11}+\alpha(n-3)A_{22})\\
& \quad-f(1-|dr|^2)-a\frac{n(n-2)}2+a\frac{n^2-4}4|dr|^2\\
\ge &-\widetilde\bric_\alpha (\tilde e_1,\tilde e_2) 
+2(n-1+\alpha(n-2))-(n+\alpha(n-1))|dr|^2\\
& \quad-\big((n-2-\alpha)dr(e_1)^2+\alpha(n-3)dr(e_2)^2\big)\\
& \quad-f(1-|dr|^2)-a\frac{n(n-2)}2+a\frac{n^2-4}4|dr|^2\\
\ge & C(|dr|^2)-\widetilde{\Lambda}_\alpha
\end{align*}
where 
\[
C(t)=2(n-1+\alpha(n-2))-(2n-2+\alpha(n-2))t-f(1-t)-a\frac{n(n-2)}2+a\frac{n^2-4}4t
\]
$C$ is an affine function and $0\le |dr|^2\le 1$, so $C(|dr|^2)\ge \min 
(C(0),C(1))$. We have
\begin{align*}
C(1)=&2(n-1+\alpha(n-2))-(2n-2+\alpha(n-2))-a\frac{n(n-2)}2+a\frac{n^2-4}4\\
&=\alpha(n-2)-a\frac{(n-2)^2}4=(n-2)(\alpha-a\frac{n-2}4)
\end{align*}
and
\begin{align*}
C(0)=2(n-1+\alpha(n-2))-f-a\frac{n(n-2)}2
\end{align*}

If we consider $a=\frac{11}{10}$ and $\alpha=\frac{40}{43}$, we have 
$a>\frac12$, 
$2a\ge \alpha$, $\alpha\le 1$ and $W=\frac{26697}{184900}>0$. So the above 
computations apply. We have
\[
C(0)=\frac{731975}{1530628}\simeq 0.47 \quad\text{and}\quad 
C(1)=\frac{543}{1720}\simeq 0.31
\]
So for these values of $a$ and $\alpha$, and with $\delta=\frac3{10}\le 
\min(C(0),C(1))$, we have
\[
V=a\left(r^2|A|^2-\frac{n(n-2)}2+\frac{n^2-4}4|dr|^2\right)\ge 
\delta-\widetilde\Lambda_\alpha
\]
By \eqref{eq:stab}, the spectral estimate \eqref{eq:specest} is true. 
Theorem~\ref{th:stable_estim} is proved.
\section{The $\mu$-bubble construction}
\label{sec:mububble}

In this section, we produce a warped $\mu$-bubble with a spectral Ricci 
curvature lower bound. So we start with a connected complete non-compact 
Riemannian 
manifold $(N^n,\bar g)$ with a spectral lower bound on the $\alpha$-bi-Ricci 
curvature: there is a smooth function $\barre V$ on $N$ such that 
\[
\barre V\ge \delta -\barre\Lambda_\alpha
\]
and 
\begin{equation}\label{eq:specbiric}
\int_N |\barre \nabla\phi|_{\bar g}^2 dv_{\bar g}\ge \int_N \frac1a \barre 
V\phi^2 
dv_{\bar 
g}
\end{equation}
for any $\phi\in C_c^1(N)$

\begin{Thm}\label{th:mububble}
Assume $(N,\bar g)$ as above with $n=5$, $a=\frac{11}{10}$,
$\alpha=\frac{40}{43}$ and $\delta=\frac3{10}$. Let $\Ome_+$ be 
a domain in $N$ (\textit{i.e.} an open subset with compact smooth boundary) 
such that $N\setminus \barre \boN_{100\pi}(\Ome_+)\neq \emptyset$. 
Then there is a domain $\Ome_*$ with 
\begin{itemize}
\item $\Ome_+\subset \Ome_*\subset \barre\boN_{100\pi}(\Ome_+)$ and
\item there is a smooth function $V$ on $\Sigma=\partial\Ome_*$ such that
\[
V\ge \frac\delta2-\alpha\lambda^\Sigma
\]
and
\begin{equation}\label{eq:specric2}
\frac4{4-a}\int_\Sigma|\nabla\phi|^2dv_g\ge \int_\Sigma V\phi^2dv_g
\end{equation}
for any $\phi\in C^1(\Sigma)$ where $g$  is the induced metric on $\Sigma$.
\end{itemize}
\end{Thm}

\subsection{Construction of the $\mu$-bubble}

Because of the spectral control \eqref{eq:specbiric} on $N$, we know (see 
\cite{FCSc}) that there is 
a positive function $w$ on $N$ such that 
\begin{equation}\label{eq:fischer}
-a\barre \Delta w=\barre Vw\ge (\delta-\barre\Lambda_\alpha)w
\end{equation}

Let us recall quickly the construction of the $\mu$-bubble. Let $\Ome_-$ be a 
domain in $N$ such that $\Ome_+\subset\subset \Ome_-\subset 
\barre\boN_{100\pi}(\Ome_+)$. Let $h:\Ome_-\setminus\Ome_+\to \R$ be a smooth 
function  such 
that $\lim_{p\to \partial\Ome_+} h(p)=+\infty$ and $\lim_{p\to \partial \Ome_-} 
h(p)=-\infty$. Let $\underline\Ome$ be a domain with 
$\Ome_+\subset\subset \underline\Ome\subset\subset \Ome_-$.

For any sets of finite perimeter $\Ome$ with $\Ome_+\subset\subset 
\Ome\subset\subset \Ome_-$, we consider the quantity
\[
\boA(\Ome)=\int_{\partial^*\Ome}w^a-\int_U (\chi_{\Ome}-\chi_{\underline\Ome}) 
hw^a
\]
where $\partial^*\Ome$ is the reduced boundary of $\Ome$. By similar argument 
to the ones in \cite{ChoLi2,Zhu}, there
there is a set of finite perimeter $\Ome_*$ ($\Ome_+\subset\subset 
\Ome_*\subset\subset \Ome_-$) which minimize the 
functional $\boA$. Moreover its reduced boundary $\partial^*\Ome_*=\Sigma$ is 
non empty ($N\setminus \barre \boN_{100\pi}(\Ome_+)\neq \emptyset$) and 
smooth (see for 
example \cite{Mor1,Tam}). 

\subsection{Spectral Ricci-curvature bound of the $\mu$-bubble} 

We denote by $k=n-1$ the dimension of $\Sigma$ and by $\eta$ the outgoing unit 
normal to $\Sigma$.

As in \cite[Proposition~4.2]{ChLiMiSt}, if $\phi$ is a function on $\Sigma$, 
writing the first variation of $\boA$ 
for a variation $\{\Ome_t\}$ of $\Ome_*$ generated by $\phi\eta$ gives
\[
0=\frac{d}{dt}\boA(\Ome_t)_{|t=0}=\int_\Sigma 
(Hw^a+aw^{a-1}dw(\eta)-hw^a)\phi=\int_\Sigma (H+aw^{-1}dw(\eta)-h)w^a\phi
\]
Since this is true for any $\phi$,
\begin{equation}\label{eq:meancurv}
H=h-ad\ln w(\eta)
\end{equation}

Computing the second derivative of $\boA(\Ome_t)$, we obtain
\begin{align*}
0\le \frac{d^2}{dt^2}\boA(\Ome_t)_{|t=0}=\int_\Sigma 
w^a\Big(&-\phi\Delta\phi-(|B|^2+\barre\ric(\eta,\eta))\phi^2-aw^{-2}dw(\eta)^2
 \phi^2 \\
&
+aw^{-1}\barre\nabla^2w(\eta,\eta)\phi^2-aw^{-1}(\barre\nabla 
w,\nabla\phi)\phi-dh(\eta)\phi^2\Big)
\end{align*}
where $B$ is the second fundamental form of $\Sigma$. So 
\begin{align*}
0\le \int_\Sigma -\Div(w^a\phi\nabla\phi)+
w^a\Big(|\nabla\phi|^2-&(|B|^2+\barre\ric(\eta,\eta))\phi^2-aw^{-2}dw(\eta)^2
 \phi^2 \\
&
+aw^{-1}\barre\nabla^2w(\eta,\eta)\phi^2-dh(\eta)\phi^2\Big)\\
\end{align*}
Using $\barre\nabla^2w(\eta,\eta)=\barre \Delta w-\Delta w-Hdw(\eta)$, we obtain
\begin{equation}\label{eq:stabmu}
\begin{split}
0\le \int_\Sigma 
w^a\Big(|\nabla\phi|^2-&(|B|^2+\barre\ric(\eta,\eta))\phi^2-aw^{-2}dw(\eta)^2
 \phi^2 \\
&
+aw^{-1}(\barre \Delta w-\Delta w-Hdw(\eta))\phi^2-dh(\eta)\phi^2\Big)\\
\end{split}
\end{equation}
For $\phi=w^{-a/2}\psi$, we have $\nabla \phi=w^{-a/2}\nabla\psi-\frac 
a2w^{-a/2-1}\psi\nabla 
w$. So we can write
\begin{align*}
\int_\Sigma w^a(|\nabla\phi|^2-aw^{-1}\Delta 
w\phi^2)=&\int_\Sigma|\nabla\psi|^2-aw^{-1}\psi(\nabla 	
w,\nabla\psi)+\frac{a^2}4\psi^2w^{-2}|\nabla w|^2-a\psi^2w^{-1}\Delta w\\
=&\int_\Sigma |\nabla\psi|^2-a\Div(\psi^2w^{-1}\nabla w)+aw^{-1}\psi(\nabla 	
w,\nabla\psi)\\
&\qquad-(a-\frac{a^2}4)\psi^2w^{-2}|\nabla w|^2\\
=&\int_\Sigma |\nabla\psi|^2+aw^{-1}\psi(\nabla 	
w,\nabla\psi)-(a-\frac{a^2}4)\psi^2w^{-2}|\nabla w|^2\\
\end{align*}
Using that $w^{-1}\psi(\nabla w,\nabla\psi)\le \eps 
|\nabla\psi|^2+\frac1{4\eps}\psi^2w^{-2}|\nabla w|^2$ with $\eps=\frac1{4-a}$,
we get
\[
\int_\Sigma w^a(|\nabla\phi|^2-aw^{-1}\Delta 
w\phi^2)\le \frac{4}{4-a}\int_\Sigma |\nabla \psi|^2
\]
From \eqref{eq:stabmu} and \eqref{eq:fischer}, we then obtain
\begin{equation}\label{eq:stabmu2}
\begin{split}
\frac4{4-a}\int_\Sigma|\nabla\psi|^2\ge&\int_\Sigma 
\Big(|B|^2+\barre\ric(\eta,\eta)+aw^{-2}dw(\eta)^2-aw^{-1}\barre\Delta w+aHd\ln 
(\eta)\Big)\psi^2\\
&\qquad\qquad+adh(\eta)\psi^2\\
\ge 
&\int_\Sigma\Big(|B|^2+\barre\ric(\eta,\eta)+\delta-\barre\Lambda_\alpha 
+aw^{-2}dw(\eta)^2+aHd\ln 
(\eta)\Big)\psi^2\\
&\qquad\qquad+adh(\eta)\psi^2\\
\end{split}
\end{equation}

Let $(e_1,\dots,e_k)$ be an orthonormal basis of $\Sigma$. Using Gauss 
equation, 
we have
\begin{align*}
\alpha\ric^\Sigma(e_1,e_1)=\alpha\sum_{j=2}^kR_{1jj1}^\Sigma=& 
\alpha\sum_{j=2}^k(\barre R_{1jj1}+B_{11}B_{jj}-B_{1j}^2)\\
=&\barre\bric_\alpha(\eta,e_1)-\barre\ric(\eta,\eta)+\alpha\sum_{j=2}^k 
(B_{11}B_{jj}-B_{1j}^2)
\end{align*}
So assuming that $e_1$ is such that $\ric^\Sigma(e_1,e_1)=\lambda^\Sigma$, we 
have
\[
\barre\ric(\eta,\eta)-\barre\Lambda_\alpha\ge 
\barre\ric(\eta,\eta)-\barre\bric_\alpha(\eta,e_1)=-\alpha 
\lambda^\Sigma+\alpha 
\sum_{j=2}^k (B_{11}B_{jj}-B_{1j}^2)
\]
So the above inequality and using $\tr B=H$ in \eqref{eq:stabmu2}, we then get 
the inequality
\begin{align*}
\frac4{4-a}\int_\Sigma|\nabla\psi|^2\ge&\int_\Sigma 
\psi^2\Big(\delta-\alpha\lambda^\Sigma +|B|^2+\alpha \sum_{j=2}^k 
(B_{11}B_{jj}-B_{1j}^2) 
+a(d\ln w(\eta))^2\\
&\qquad\quad+aHd\ln w(\eta)+adh(\eta)\Big)\\
\ge  &\int_\Sigma 
\psi^2\Big(\delta-\alpha\lambda^\Sigma +|B|^2+\alpha HB_{11}-\alpha 
\sum_{j=1}^k 
B_{1j}^2
+a(d\ln w(\eta))^2\\
&\qquad\quad+aHd\ln w(\eta)+adh(\eta)\Big)\\
\end{align*}
Using \eqref{eq:meancurv}, we have
\begin{align*}
K:=|B|^2+\alpha HB_{11}-\alpha 
\sum_{j=1}^k 
B_{1j}^2
&+a(d\ln w(\eta))^2
+aHd\ln w(\eta)=\\
&|B|^2+\alpha HB_{11}-\alpha \sum_{j=1}^k B_{1j}^2 +\frac1a(H-h)^2+H(h-H)
\end{align*}
Let us denote by $\Phi$ the traceless part of $B$ and let $\Phi_\Delta$ denote 
the vector $(\Phi_{11},\dots,\Phi_{kk})\in F_k$. Thus, for $\alpha\le 2$, we 
have
\[
K\ge \frac1kH^2+|\Phi_\Delta|^2+\frac\alpha k H^2+\alpha 
H\Phi_{11}-\alpha(\frac1kH+\Phi_{11})^2+
\frac1a(H-h)^2+H(h-H)
\]
We can write a decomposition of $\Phi_\Delta$ in an orthonormal basis of $F_k$
\[
\Phi_\Delta=\sum_{i=1}^{k-2}\begin{pmatrix}
0\\E_i
\end{pmatrix}x_i+\frac1{\sqrt{k(k-1)}}\begin{pmatrix}k-1\\-1\\ \vdots\\-1 
\end{pmatrix}z
\]
where $(E_i)_{1\le i\le k-2}$ is an orthonormal basis of $F_{k-1}$.
We then have
\begin{align*}
K\ge & \frac1kH^2+z^2+\frac\alpha k H^2+\alpha 
H\sqrt{\frac{k-1}k}z-\alpha(\frac1kH+\sqrt{\frac{k-1}k}z)^2+
\frac1a(H-h)^2+H(h-H)\\
\ge &(\frac1k+\frac\alpha 
k-\frac{\alpha}{k^2}+\frac1a-1)H^2+(1-\alpha\frac{k-1}k)z^2+\frac1a h^2 
+\alpha\sqrt{\frac{k-1}k}(1-\frac2k)Hz+(1-\frac2a)Hh
\end{align*}
The above expression is a quadratic form in $(H,z,h)$ associated to the matrix
\[
G=\begin{pmatrix}
\frac1k+\frac\alpha 
k-\frac{\alpha}{k^2}+\frac1a-1&\frac\alpha2\sqrt{\frac{k-1}k}(1-\frac2k)& 
\frac12-\frac1a\\
\frac\alpha2\sqrt{\frac{k-1}k}(1-\frac2k)& 1-\alpha\frac{k-1}k&0\\
\frac12-\frac1a&0&\frac1a
\end{pmatrix}
\]
Notice that this matrix is positive definite if $1-\alpha\frac{k-1}k>0$ and 
$\det(G)>0$. Actually for $k=4$, $a=\frac{11}{10}$, $\alpha=\frac{40}{43}$, we 
have $1-\alpha\frac{k-1}k=\frac{13}{43}>0$ and
\[
\det\Big(G-\begin{pmatrix}
0&0&0\\0&0&0\\0&0&\frac1{22}
\end{pmatrix}\Big)=\frac{2599}{1789832}>0
\]
So $K\ge \frac1{22}h^2$. Finally, for our values of the parameters, we have
\begin{equation}\label{eq:specric}
\frac4{4-a}\int_\Sigma|\nabla\psi|^2\ge 
\int_\Sigma 
\psi^2\big(\frac\delta2-\alpha\lambda^\Sigma\big)+\psi^2\big(\frac\delta2 
+\frac1{22}h^2+adh(\eta)\big)
\end{equation}

\subsection{End of the proof}

We need to choose the domain $\Ome_-$ and the function $h$. Let 
$\Psi:N\setminus \Ome_+\to \R_+$ be a smoothing of the distance function 
$d_{\bar 
g}(\cdot,\partial\Ome_+)$ such that $\frac12 d_{\bar 
g}(\cdot,\partial\Ome_+)\le \Psi\le 
2 d_{\bar 
g}(\cdot,\partial\Ome_+)$ and $|\barre\nabla \Psi|_{\bar g}\le 2$. Let 
$\eps>0$ small such 
that $(1+\eps)11\pi\sqrt{\frac{88}{15}}$ is a regular value of $\Psi$. Let us 
define $\Ome_-=\Ome_+\cup\{\Phi\le (1+\eps)11\pi\sqrt{\frac{88}{15}}\}$. On 
$\Ome_-$, $d_{\bar 
g}(\cdot,\partial\Ome_+)\le 2(1+\eps)11\pi\sqrt{\frac{88}{15}}\le 
100\pi$, so $\Ome_-\subset \barre\boN_{100\pi}(\Ome_+)$.

On $\{0<\Psi< (1+\eps)11\pi\sqrt{\frac{88}{15}}\}$, we consider the function 
$h$ defined by $h=k\circ\frac{\Psi}{1+\eps}$ where
\[
k(t)=-\sqrt{\frac{33}{10}}\tan(\frac1{11}\sqrt{\frac{15}{88}}t-\frac{\pi}2)
\]
for $t\in(0,11\pi\sqrt{\frac{88}{15}})$. We have 
$\lim_{p\to\partial\Ome_\pm}h(p)=\pm\infty$. Notice that $k$ solves 
$-k'=\frac3{44}+\frac5{242} k^2$ so
\[
|adh(\eta)|=a|k'(\frac{\Psi(p)}{1+\eps})|\frac{|\Psi'(p)|}{1+\eps}\le 
\frac{2a}{1+\eps}(\frac3{44}+\frac5{242}h^2)\le\frac3{20}+\frac1{22}h^2 
=\frac\delta2+\frac1{22} h^2
\]
Hence, the above construction applies and, for our choices of parameters, 
\eqref{eq:specric} becomes
\[
\frac4{4-a}\int_\Sigma|\nabla\psi|^2\ge 
\int_\Sigma 
\psi^2\big(\frac\delta2-\alpha\lambda^\Sigma\big)
\]
This ends the proof of Theorem~\ref{th:mububble}.
\section{Stable Bernstein problem}
\label{sec:proof}

In this section we prove Theorem~\ref{th:main}. This a consequence of the 
following volume growth estimate. 

\begin{Prop}\label{prop:volestim} Let $F:M^5\looparrowright \R^6$ be a 
complete, immersed, two-sided, simply-connected stable minimal hypersurface. 
Let $\boB_\rho$ denote the geodesic ball of radius $\rho>0$ centered at some 
point $p_0$ in $M$ (for the induced metric $g$). Then
\[
\vol(\boB_\rho)\le 
\vol(\B^5)(\frac{800}{43})^{5/2}\big(2\exp(100\pi)\big)^5\rho^5
\]
\end{Prop}

\begin{proof}
First, up to a translation, we may assume that $F(p_0)=0$.
Let $\Ome_+$ be a smooth compact domain in $M$ such that $\boB_\rho\subset 
\Ome_+\subset \boB_{2\rho}$ and such that $0\notin F(\partial\Ome_+)$. We 
consider the Gulliver-Lawson conformal metric $\tilde g=r^{-2}g$. By 
Theorem~\ref{th:stable_estim} and 
Theorem~\ref{th:mububble}, there is $\Ome_*$ a domain in $M$ such that 
$\Ome_+\subset \Ome_*\subset \widetilde\boN_{100\pi}(\Ome_+) 
$ and $\partial\Ome_*$ satisfies the spectral Ricci lower bound 
\eqref{eq:specric2} for the metric induced by $\tilde g$.

By \cite[Theorem~1]{CaShZh}, $M$ has  one end. We consider $\Ome_{**}$ the 
connected component of $\Ome_*$ that contains $\boB_\rho$. We assume 
$M$ is simply 
connected so the unbounded component of $M\setminus \Ome_{**}$ 
has only boundary component $\Sigma_0$. Let $\Ome'$ be the bounded component of 
$M\setminus \Sigma_0$. We have $\boB_\rho\subset \Ome'$ and 
$\partial\Ome'\subset \widetilde\boN_{100\pi}(\boB_{2\rho})$.

On $\partial\boB_{2\rho}$, the Euclidean distance function $r$ is bounded by 
$2\rho$. So, by \cite[Lemma~6.2]{ChoLi}, on 
$\widetilde\boN_{100\pi}(\boB_{2\rho	})$, the Euclidean 
distance function $r$ is bounded by $2\rho\exp(100\pi)$.

Now, because of the spectral Ricci lower bound \eqref{eq:specric2} and since 
$\frac 
4{(4-a)\alpha}=\frac{43}{29}<\frac32= \frac{k-1}{k-2}$, we can apply the 
volume estimate of Antonelli and Xu \cite[Theorem~1]{AnXu} for the metric 
$\tilde g$ and obtain
\[
\vol_{\tilde g}(\Sigma_0)\le 
(\frac{\delta}{6\alpha})^{-2}\vol(\S^4)=(\frac{800}{43})^2\vol(\S^4)
\]
So scaling back to the Euclidean metric
\[
\vol(\Sigma_0)\le 
(\frac{800}{43})^2\vol(\S^4)\big(2\exp(100\pi)\big)^4\rho^4
\]

Finally we can apply the isoperimetric inequality for minimal hypersurfaces in 
$\R^{n+1}$ \cite{Bre2,MiSi} to obtain
\[
\vol_g(\boB_\rho)\le \vol_g(\Ome')\le 
\vol(\B^5)(\frac{800}{43})^{5/2}\big(2\exp(100\pi)\big)^5\rho^5
\]
\end{proof}

\begin{proof}[Proof of Theorem~\ref{th:main}]
Let $M\looparrowright \R^6$ be an immersed, connected,complete, two-sided, 
stable minimal 
hypersurface. The stability assumption lifts to the universal cover, so we can 
assume $M$ to be simply connected. By Proposition~\ref{prop:volestim}, $M$ has 
Euclidean volume growth. So by \cite{ScSiYa} (see also \cite{Bel}), we obtain 
that $M$ is a flat 
hyperplane.
\end{proof}


\begin{thebibliography}{10}

\bibitem{Alm2}
F.~J. Almgren, Jr.
\newblock Some interior regularity theorems for minimal surfaces and an
  extension of {B}ernstein's theorem.
\newblock {\em Ann. of Math. (2)}, 84:277--292, 1966.

\bibitem{AnXu}
Gioacchino Antonelli and Kai Xu.
\newblock New spectral {B}ishop-{G}romov and {B}onnet-{M}yers theorems and
  application to isoperimetry.
\newblock preprint, arXiv:2405.08918.

\bibitem{Bel}
Costante Bellettini.
\newblock Extensions of {S}choen--{S}imon--{Y}au and {S}choen--{S}imon theorems
  via iteration \`a la {D}e {G}iorgi.
\newblock preprint, arXiv:2310.01340.

\bibitem{Ber}
Serge Bernstein.
\newblock \"{U}ber ein geometrisches {T}heorem und seine {A}nwendung auf die
  partiellen {D}ifferentialgleichungen vom elliptischen {T}ypus.
\newblock {\em Math. Z.}, 26:551--558, 1927.

\bibitem{BoGiGi}
E.~Bombieri, E.~De~Giorgi, and E.~Giusti.
\newblock Minimal cones and the {B}ernstein problem.
\newblock {\em Invent. Math.}, 7:243--268, 1969.

\bibitem{Bre2}
Simon Brendle.
\newblock The isoperimetric inequality for a minimal submanifold in {E}uclidean
  space.
\newblock {\em J. Amer. Math. Soc.}, 34:595--603, 2021.

\bibitem{CaShZh}
Huai-Dong Cao, Ying Shen, and Shunhui Zhu.
\newblock The structure of stable minimal hypersurfaces in {${\bf R}^{n+1}$}.
\newblock {\em Math. Res. Lett.}, 4:637--644, 1997.

\bibitem{CaMaRo}
Giovanni Catino, Paolo Mastrolia, and Alberto Roncoroni.
\newblock Two rigidity results for stable minimal hypersurfaces.
\newblock {\em Geom. Funct. Anal.}, 34:1--18, 2024.

\bibitem{ChoLi3}
Otis Chodosh and Chao Li.
\newblock Stable minimal hypersurfaces in {${\bf R}^4$}.
\newblock preprint, arXiv:2108.11462.

\bibitem{ChoLi}
Otis Chodosh and Chao Li.
\newblock Stable anisotropic minimal hypersurfaces in {${\bf R}^4$}.
\newblock {\em Forum Math. Pi}, 11:Paper No. e3, 22, 2023.

\bibitem{ChoLi2}
Otis Chodosh and Chao Li.
\newblock Generalized soap bubbles and the topology of manifolds with positive
  scalar curvature.
\newblock {\em Ann. of Math. (2)}, 199:707--740, 2024.

\bibitem{ChLiMiSt}
Otis Chodosh, Chao Li, Paul Minter, and Douglas Stryker.
\newblock Stable minimal hypersurfaces in $\mathbf{R}^5$.
\newblock preprint, arXiv:2401.01492.

\bibitem{ChLiSt}
Otis Chodosh, Chao Li, and Douglas Stryker.
\newblock Complete stable minimal hypersurfaces in positively curved
  $4$-manifolds.
\newblock preprint, arXiv:2202.07708.

\bibitem{DeG}
Ennio De~Giorgi.
\newblock Una estensione del teorema di {B}ernstein.
\newblock {\em Ann. Scuola Norm. Sup. Pisa Cl. Sci. (3)}, 19:79--85, 1965.

\bibitem{DoPe}
M.~do~Carmo and C.~K. Peng.
\newblock Stable complete minimal surfaces in {${\bf R}^{3}$} are planes.
\newblock {\em Bull. Amer. Math. Soc. (N.S.)}, 1:903--906, 1979.

\bibitem{FCSc}
Doris Fischer-Colbrie and Richard Schoen.
\newblock The structure of complete stable minimal surfaces in {$3$}-manifolds
  of nonnegative scalar curvature.
\newblock {\em Comm. Pure Appl. Math.}, 33:199--211, 1980.

\bibitem{Fle}
Wendell~H. Fleming.
\newblock On the oriented {P}lateau problem.
\newblock {\em Rend. Circ. Mat. Palermo (2)}, 11:69--90, 1962.

\bibitem{GuLa}
Robert Gulliver and H.~Blaine Lawson, Jr.
\newblock The structure of stable minimal hypersurfaces near a singularity.
\newblock In {\em Geometric measure theory and the calculus of variations
  ({A}rcata, {C}alif., 1984)}, volume~44 of {\em Proc. Sympos. Pure Math.},
  pages 213--237. Amer. Math. Soc., Providence, RI, 1986.

\bibitem{HoYa}
Han Hong and Zeitan Yan.
\newblock Rigidity and nonexistence of {CMC} hypersurfaces in $5$-manifolds.
\newblock preprint, arXiv:2405.06867.

\bibitem{MiSi}
J.~H. Michael and L.~M. Simon.
\newblock Sobolev and mean-value inequalities on generalized submanifolds of
  {$R\sp{n}$}.
\newblock {\em Comm. Pure Appl. Math.}, 26:361--379, 1973.

\bibitem{Mor1}
Frank Morgan.
\newblock Regularity of isoperimetric hypersurfaces in riemannian manifolds.
\newblock {\em Trans. Amer Math. Soc.}, 355:5041--5052, 2003.

\bibitem{Pog}
A.~V. Pogorelov.
\newblock On the stability of minimal surfaces.
\newblock {\em Dokl. Akad. Nauk SSSR}, 260:293--295, 1981.

\bibitem{ScSiYa}
R.~Schoen, L.~Simon, and S.~T. Yau.
\newblock Curvature estimates for minimal hypersurfaces.
\newblock {\em Acta Math.}, 134:275--288, 1975.

\bibitem{ScSi}
Richard Schoen and Leon Simon.
\newblock Regularity of stable minimal hypersurfaces.
\newblock {\em Comm. Pure Appl. Math.}, 34:741--797, 1981.

\bibitem{ShYe}
Ying Shen and Rugang Ye.
\newblock On the geometry and topology of manifolds of positive bi-ricci
  curvature.
\newblock preprint, arXiv:9708014.

\bibitem{Sims}
James Simons.
\newblock Minimal varieties in riemannian manifolds.
\newblock {\em Ann. of Math. (2)}, 88:62--105, 1968.

\bibitem{Tam}
Italo Tamanini.
\newblock Regularity results for almost minimal oriented hypersurfaces in
  ${R^N}$.
\newblock {\em Quaderni del Dipartimento di Matematica dell’ Università di
  Lecce}, 1, 1984.

\bibitem{Zhu}
Jintian Zhu.
\newblock Width estimate and doubly warped product.
\newblock {\em Trans. Amer. Math. Soc.}, 374:1497--1511, 2021.

\end{thebibliography}
\end{document}